\documentclass[12pt,reqno]{amsart}

\usepackage[latin1]{inputenc}
\usepackage{amsmath}
\usepackage{amsfonts}
\usepackage{amssymb}
\usepackage{graphics}
\usepackage{color}

\usepackage{enumerate}
\usepackage{amssymb,amsmath,amsthm,amscd}
\usepackage{latexsym,verbatim,graphicx,amsfonts}

\usepackage{amscd}
\usepackage{amsmath}
\usepackage{amssymb}
\usepackage{amsthm}
\usepackage{latexsym}
\usepackage{verbatim}
\usepackage{hyperref}

\theoremstyle{plain}
\newtheorem{theorem}{Theorem}[section]
\newtheorem{thm}[theorem]{Theorem}
\newtheorem{prob}[theorem]{Problem}

\theoremstyle{definition}

\newtheorem{ex}[theorem]{Example}
\newtheorem{OQ}{Open Question}
\theoremstyle{remark}

\newcommand{\R}{\mathbb{R}}
\newcommand{\C}{\mathbb{C}}

\newcommand{\T}{\mathbb{T}}
\newcommand{\D}{\mathbb{D}}
\newcommand{\N}{\mathbb{N}}

\subjclass[2010]{Primary 30Hxx; Secondary 30Jxx. Key words: optimal polynomial approximants, digital filters, Dirichlet spaces, reproducing kernels, orthogonal polynomials}

\begin{document}
\title[Survey of optimal approximants]{A Survey of optimal polynomial approximants, applications to digital filter design, and related open problems}
\author[B\'en\'eteau]{Catherine B\'en\'eteau}
\address{Department of Mathematics, University of South Florida, 4202 E. Fowler Avenue,
Tampa, Florida 33620-5700, USA.} \email{cbenetea@usf.edu}
\author[Centner]{Raymond Centner}
\address{Department of Mathematics, University of South Florida, 4202 E. Fowler Avenue,
Tampa, Florida 33620-5700, USA.} \email{rcentner@usf.edu}

\begin{abstract}
In the last few years, the notion of optimal polynomial approximant has appeared in the mathematics literature in connection with Hilbert spaces of analytic functions of one or more variables.  In the 70s, researchers in engineering and applied mathematics introduced least-squares inverses in the context of digital filters in signal processing. It turns out that in the Hardy space $H^2$ these objects are identical. This paper is a survey of known results about optimal polynomial approximants.  In particular, we will examine their connections with orthogonal polynomials and reproducing kernels in weighted spaces and digital filter design.  We will also describe what is known about the zeros of optimal polynomial approximants, their rates of decay, and convergence results.  Throughout the paper, we state many open questions that may be of interest.   
\end{abstract}

\maketitle

\section{Introduction}\label{Intro}

Consider a collection of Hilbert spaces of analytic functions of the unit disk $\D$ defined as follows.  Given weights
$\{\omega_k\}_{k \in \N}$ satisfying $\omega_0=1$, $\omega_k >0$
and $\omega_k/\omega_{k+1} \rightarrow 1$ as $k \rightarrow \infty$,
we define the weighted Hardy space $H^2_\omega$ as the space of functions  $f(z)=\sum_{k=0}^{\infty} a_k z^k$ analytic in $\D$ such that
\begin{equation*} 
	\|f\|^2_\omega= \sum_{k=0}^{\infty} |a_k|^2 \omega_k < \infty.
\end{equation*}
$H^2_\omega$ is a Hilbert space with inner product
\begin{equation*}
	\langle f, g \rangle_\omega = \sum_{k=0}^{\infty} a_k \overline{b_k} \omega_k,
\end{equation*}
for $g(z)=\sum_{k=0}^{\infty} b_k z^k.$  If $\omega_k = (k+1)^{\alpha}$ for a real parameter $\alpha,$ the spaces $H^2_\omega= {\mathcal{D}}_{\alpha}$ are called Dirichlet-type spaces and 
include the Hardy space $H^2 = \mathcal{D}_0,$ the classical Dirichlet space $\mathcal{D}_1 = \mathcal{D},$ and the Bergman space $\mathcal{D}_{-1} =A^2.$ In the case of the Hardy space $H^2$, we omit the subscript $\omega$ and write $\langle\cdot,\cdot\rangle$ for the inner product. Given a function $f \in H^2_\omega$ not identically $0$ and $n \in \N,$ we define the $n$-th optimal polynomial approximant (opa) of $1/f$ in $H^2_\omega$ to be the polynomial $q_n$ that minimizes
$\|pf-1\|_{\omega}$ among all polynomials $p$ of degree at most $n.$ Notice that if we denote by $\mathcal{P}_n$ to be the set of all polynomials of degree at most $n$, then $f \cdot\mathcal{P}_n$ is a finite dimensional subspace of $H^2_\omega$, and $q_n f$ is simply the projection of $1$ onto that subspace.  Therefore, $q_n$ always exists and is unique. 

These polynomials were defined in \cite{BCLSS} in the spaces $\mathcal{D}_{\alpha}$ in connection with cyclicity: a function $f \in H^2_\omega$ is called cyclic if its polynomial multiples are dense in the space.  In particular, if $f$ is cyclic, there must be a sequence of polynomial multiples of $f$ that approximate $1$, and thus, the opas $q_n$ are the ones that achieve this approximation at the fastest rate in terms of the degree. On the other hand, if $1$ can be approximated by polynomial multiples of $f$, then so can any polynomial, and since polynomials are dense in $H^2_\omega$, $f$ must be cyclic.  Therefore $f$ is cyclic if and only if the opas $q_n$ satisfy $\|q_nf-1\|_\omega \rightarrow 0$ as $n \rightarrow \infty.$  Thus, one can view cyclicity through the lens of the behavior of the optimal polynomial approximants.

The authors of \cite{BCLSS} were hoping to be able to shed some light on the Brown and Shields conjecture (see \cite{BS}) using opas.  The Brown and Shields conjecture states that for a function $f$ in the classical Dirichlet space $\mathcal{D},$ if $f$ is outer and has boundary zero set $Z(f)$ with logarithmic capacity $0$, then $f$ is cyclic.  (The converse is known to be true.) Although this initial hope has not yet proved successful, many other interesting problems related to optimal polynomial approximants  arose, in particular in connection with inner functions (see, e.g., \cite{Remarks,Fe,SecoTellez}).

It turns out that optimal polynomial approximants had been introduced much earlier in the engineering literature in the 70s, in connection with signal processing and digital filter construction.  They were first introduced by Robinson (\cite{Ro}) and referred to as least-squares inverses. Robinson wanted to find the inverse under discrete convolution of a finite-length minimum-delay wavelet that best approximates the unit spike.  From a mathematical point of view, this meant, given a finite sequence $b:=(b_0,b_1, \ldots, b_n)$ of real numbers, where the largest coefficient in magnitude is $b_0$ (the so-called ``minimum delay"), find a ``wavelet", i.e., a sequence $a:= (a_0,a_1, \ldots, a_m)$ of real numbers such that the difference between the convolution $a*b$ and the unit spike $(1,0,\ldots, 0)$ has smallest $\ell^2$-norm. This is equivalent to the problem of finding the optimal polynomial approximant of $1/f$ for a polynomial $f$ with $f(0)\neq 0$. In \cite{Ch}, Chui examined the \emph{double least-squares inverse} of a polynomial $f$ of degree $n$ by first finding the $k$-th opa $q_k$ of $1/f$ and then finding the $n$-th opa $p_{n,k}$ of $1/q_k$. He studied the zeros of the double least-squares inverses, noticed the connection with orthogonal polynomials (which had also been discussed for polynomials with real coefficients in \cite{GK}), and proved that $p_{n,k}$ converges to $f$ in the finite-dimensional space of polynomials of degree $n$ if and only if $f$ has no zeros in the open unit disk. He went on to discuss some generalizations to the case when $f \in H^2,$ which was further studied by Izumino in \cite{Iz}. In several variables, least-squares inverses were studied by a number of authors in the 70s in connection with recursive digital filters (see, e.g., \cite{GK,STJ}). Two recent papers of Sargent and Sola (\cite{SS,SS2}) discuss some of these developments and examine the more complicated relationship between optimal polynomial approximants and orthogonal polynomials in several variables.

This paper is a survey of some of the known results about optimal polynomial approximants, applications to digital filter theory, and related open problems.  One of the goals of the paper is, as also discussed in \cite{SS}, to connect the two sets of literature on these topics. In Section 2, we describe how to compute opas, give some basic examples, and explain how opas are connected to orthogonal polynomials and reproducing kernels for weighted spaces.  In Section 3, we discuss how opas were used in the design of digital filters in the engineering literature of the 70s.  In Section 4, we survey some known results about the zeros of opas in the single and multivariable case and state several open questions.  Finally, in Section 5, we discuss rates of decay and pointwise convergence results. 

\section{Existence of opas and connections with orthogonal polynomials and reproducing kernels}

Let us now turn to a discussion of how to compute opas and their connection with classical objects in analysis.  We have already seen that for any function $ f \in H^2_\omega\setminus\{0\},$ and any $ n \in \N,$ the $n$-th opa $q_n$ always exists.  It is also not hard to see that the coefficients of $q_n$ can be computed by a set of linear equations (\cite{BCLSS,FMS}), as stated in the following theorem. 

\begin{thm}\label{LinearSys}
For any function $f\in H^2_\omega\setminus\{0\}$, let $q_n(z)=\sum_{j=0}^{n}a_jz^j$ denote the $n$-th opa of $1/f$. Define the vectors $a=\begin{bmatrix}
a_0&a_1&\dots&a_n
\end{bmatrix}^T$ and $y=\begin{bmatrix}
\overline{f(0)}&0&\dots&0
\end{bmatrix}^T$. If $B$ is the $(n+1)\times(n+1)$ matrix given by $B_{jk}=\langle z^kf,z^jf\rangle_\omega$, where $0\leq j,k\leq n$, then $$Ba=y.$$ 
\end{thm}

\begin{proof}
Note that by definition, $q_n$ is the projection of $1$ onto $\mathcal{P}_n \cdot f$.  By orthogonality in $H^2_\omega$, we thus have that
\begin{align*}
    0&=\langle q_nf-1,z^jf\rangle_\omega\\
    &=\langle\sum_{k=0}^na_kz^kf-1,z^jf\rangle_\omega\\
    &=\sum_{k=0}^na_k\langle z^kf,z^jf\rangle_\omega-\langle 1,z^jf\rangle_\omega\\
    &=\sum_{k=0}^na_k\langle z^kf,z^jf\rangle_\omega-\delta_{j,0}\overline{f(0)}
\end{align*}
for $j=0,\cdots,n$.  This is equivalent to the assertion $Ba=y$.

\end{proof}

Let's have a look at an example for $H^2$ in which the form of the matrix is simple enough to allow us to determine an explicit formula for the coefficients. 

\begin{ex}\label{1-z}

Let $q_n$ denote the $n$-th opa of $1/f$, where $f(z)=1-z$. Then $$q_n(z)=\sum_{j=0}^n\frac{n+1-j}{n+2}z^j.$$ 

\end{ex} 

\begin{proof} 
Let $q_n(z)=\sum_{j=0}^{n}a_jz^j$ and let $a$ and $y$ denote the vectors $a=
\begin{bmatrix}
a_0&a_1&a_2&\dots&a_n
\end{bmatrix}^T$ 
and $y=
\begin{bmatrix}
1&0&0&\dots&0
\end{bmatrix}^T$. 
From Theorem \ref{LinearSys}, we have 
\begin{equation*}
    Ba=y, 
\end{equation*}
where $B$ is the $(n+1)\times(n+1)$ matrix with
\begin{align*}
    B_{j,k}&=\langle z^kf,z^jf\rangle\\
    &=\langle z^k(1-z),z^j(1-z)\rangle\\
    &=\langle z^k-z^{k+1},z^j-z^{j+1}\rangle\\
    &=\langle z^k,z^j\rangle +\langle z^{k+1},z^{j+1}\rangle -\langle z^{k+1},z^j\rangle -\langle z^k, z^{j+1}\rangle
\end{align*}
for $0\leq j,k\leq n$. Therefore, we see that 
\[
B_{j,k}=
\begin{cases}
2 & \text{if $j=k$}\\
-1&\text{if $|j-k|=1$}\\
0&\text{otherwise}
\end{cases}
\]
and note that $a_0=1/2$ whenever $n=0$.  If $n=1$, we get
\[
\begin{cases}
2a_0-a_1=1\\
-a_0+2a_1=0.
\end{cases}
\]
For $n\geq 2$, the coefficients $\{a_j\}_{j=0}^n$ satisfy
\[
\begin{cases}
2a_0-a_1=1\\
-a_j+2a_{j+1}-a_{j+2}=0\qquad 0\leq j\leq n-2\\
-a_{n-1}+2a_n=0.
\end{cases}
\]
These conditions can be summarized by the matrix equation
\[
\begin{bmatrix}
2&-1&0&0&\dots&0\\
-1&2&-1&0&\dots&0\\
0&-1&2&-1&\dots&0\\

\vdots&&\ddots&\ddots&\ddots&\vdots\\

0&\dots&0&-1&2&-1\\
0&\dots&0&0&-1&2
\end{bmatrix}
\begin{bmatrix}
a_0\\
a_1\\
a_2\\
a_3\\
\vdots\\
a_n
\end{bmatrix}
=
\begin{bmatrix}
1\\
0\\
0\\
0\\
\vdots\\
0
\end{bmatrix}.
\]
It is straightforward to check that $\{\frac{n+1-j}{n+2}\}_{j=0}^n$ satisfies these conditions.

\end{proof} 

This computation was generalized in \cite{FMS}, where the authors show that the opa $q_n$ of $1/(1-z)$ in $H^2_{\omega}$ is equal to 
\[ q_n (z) = \sum_{k=0}^n \left( 1 - \frac{\sum_{j=0}^k \frac{1}{\omega_j}}{\sum_{j=0}^{n+1} \frac{1}{\omega_j}} \right). \] In $H^2,$ formulas for the opa of $1/(1-z)^a$ where $Re(a) >0$ are obtained in \cite{blSimanek}, but in general, it is very difficult to calculate opas explicitly.  Computations of opas focused on the case when $f$ is a polynomial are considered in \cite{BMS}.
From the point of view of classical analysis, it is easy to see that one way to calculate $q_n f$ is to consider the set of vectors $\{f, zf, z^2f, \ldots, z^nf\},$ use the Gram-Schmidt process to find an orthonormal basis of the finite dimensional space $f \cdot\mathcal{P}_n,$ and express $q_nf$ in that basis. This gives rise to orthonormal polynomials and is the content of the following theorem, discussed in \cite{BKLSS}.   

\begin{thm}[\cite{BKLSS}] 
Let $f \in H^2_\omega\setminus\{0\}$,  $n \in \N,$ and let $q_n$ be the $n$-th opa of $1/f.$ For each $ k = 0, 1, \ldots n,$ let $\varphi_k$ be a polynomial of degree $k$ such that $\{\varphi_k f\}_{k=0}^n$ is an orthonormal basis of $f \cdot\mathcal{P}_n.$ Then 
\begin{equation}\label{FSofqn}
  q_n(z) = \overline{f(0)} \sum_{k=0}^{n} \overline{\varphi_k(0)}  \varphi_k(z). 
\end{equation}
\end{thm}

\begin{proof}
    Since $q_nf \in f \cdot\mathcal{P}_n$ and $\{\varphi_k f\}_{k=0}^n$ is an orthonormal basis of $f \cdot\mathcal{P}_n,$ we can express $q_nf$ in that basis:
    \begin{equation*}
        q_n(z)f(z) = \sum_{k=0}^n \langle q_nf,\varphi_kf\rangle_{\omega} \varphi_k(z) f(z).
    \end{equation*}
    But $q_nf$ is the projection of $1$ onto $f \cdot\mathcal{P}_n,$ and therefore $\langle q_nf,\varphi_kf\rangle_{\omega} = \langle 1,\varphi_kf\rangle_{\omega}.$ By the defintion of the inner product in $H^2_\omega$, this gives
    \begin{equation*}
        q_n(z) f(z) = \sum_{k=0}^{n}  \overline{f(0)\varphi_k(0)}  \varphi_k(z) f(z),  
    \end{equation*} and therefore, \eqref{FSofqn} follows.
\end{proof}

The polynomials $\varphi_k$ can be thought of as orthonormal polynomials in the Hilbert space weighted by $f$, in the sense that 
$\langle \varphi_kf,\varphi_jf \rangle_\omega=\delta_{kj}.$ In particular, if we restrict to the classical space $H^2$, this is the same as saying that the polynomials $\varphi_k$ are orthonormal polynomials in the weighted space $H^2(\mu)$ where the measure $d\mu = \frac{1}{2\pi}|f|^2 d\theta.$  This approach was discussed in \cite{CC} for the Hardy space and in \cite{BKLSS} for Dirichlet-type spaces. 

Notice also that since $\{\varphi_k f\}_{k=0}^n$ is an orthonormal basis of $f \cdot\mathcal{P}_n$,  $\sum_{k=0}^{n}  \overline{f(0)\varphi_k(0)}  \varphi_k(z) f(z) =: K_n(z,0)$ is the reproducing kernel for $f \cdot\mathcal{P}_n$ at $0:$ this means that for any polynomial $q \in \mathcal{P}_n$, $\langle qf, K_n(\cdot,0)\rangle_\omega = q(0)f(0).$  This can also be seen directly, since 
$$\langle qf, q_n f\rangle_\omega = \langle qf, 1\rangle_\omega = q(0)f(0).$$
Thus $q_n(z) f(z) = K_n(z,0).$

In the particular case of $H^2,$ the $n$-th opa $q_n$ is even more tightly connected to the $n$-th orthonormal polynomial, as is well-known (see, e.g., \cite{Ge,BaS}). The authors of \cite{BKLSS} used this connection to express the opas in terms of weighted orthonormal polynomials.  
\begin{thm}[\cite{BKLSS}]\label{H2OP}
Let $f \in H^2\setminus\{0\},$  $n \in \N,$ and let $q_n$ be the $n$-th opa of $1/f.$ For each $ k = 0, 1, \ldots n,$ let $\varphi_k$ be a polynomial of degree $k$ such that $\{\varphi_k f\}_{k=0}^n$ is an orthonormal basis of $f \cdot\mathcal{P}_n.$ Let $\varphi^*_n(z) = z^n \overline{\varphi_n(1/\bar{z})}.$ Then $q_n(z)=\overline{f(0)}\hat{\varphi}_n(n)\varphi_n^*(z).$
\end{thm}

\begin{proof}
Let $q_n(z)=\sum_{k=0}^na_kz^k$ and consider the weighted space $H^2(\mu)$, where $d\mu=\frac{1}{2\pi}|f|^2d\theta$. In light of Theorem $\ref{LinearSys}$, we see that 
\begin{align*}
    \delta_{j,0}f(0)&=\sum_{k=0}^n\langle z^jf,z^kf\rangle \overline{a_k}\\
    &=\langle z^jf,\sum_{k=0}^na_kz^kf\rangle\\
    &=\langle z^j,q_n\rangle_\mu\\
    &=\langle z^n\overline{q_n},z^{n-j}\rangle_\mu\\
    &=\langle q_n^*,z^{n-j}\rangle_\mu.
\end{align*}
Now since $\{\varphi_k\}_{k=0}^n$ is an orthonormal basis of $H^2(\mu)$, it follows that
\begin{align*}
    q_n^*(z)&=\sum_{k=0}^n\langle q_n^*,\varphi_k\rangle_\mu\varphi_k(z)\\
    &=\langle q_n^*,\varphi_n\rangle_\mu\varphi_n(z)\\
    &=f(0)\overline{\hat{\varphi}_n(n)}\varphi_n(z).
\end{align*}
Therefore,
\begin{align*}
    q_n(z)&=z^n\overline{q_n^*(1/\overline{z})}\\
    &=\overline{f(0)}\hat{\varphi}_n(n)z^n\overline{\varphi_n(1/\overline{z})}\\
    &=\overline{f(0)}\hat{\varphi}_n(n)\varphi_n^*(z).
\end{align*}
\end{proof}
It is known from the theory of orthogonal polynomials that the orthonormal polynomials $\varphi_n$ have all their zeros inside the open unit disk $\D,$ and therefore the opas $q_n$ have no zeros in the closed disk. In particular, this connection was important for designing what are called stable filters, which we discuss in the next section.


\section{Digital Filter Design}\label{design}

\subsection{Introduction to filters}\label{intro}

Several problems in engineering ultimately depend on a system's response to an input.  In the case of a digital system, the input is given as a sampling sequence $\{x(n)\}_{n=-\infty}^{\infty}$, and the output $\{y(n)\}_{n=-\infty}^{\infty}$ can often be described by a difference equation
\begin{equation}
 y(n)=\sum_{k=0}^M b_k x(n-k)-\sum_{j=1}^N a_j y(n-j), \label{diff}   
\end{equation}
where the coefficients $a_j$ and $b_k$ are real numbers.  
  
If the coefficients remain constant over time, the system is known as \textit{linear time-invariant}, or LTI.  If the $a_j$'s are not all zero, then the system is referred to as \textit{recursive}.  This means that one or more of the system's output is used as an input. Now, if we consider an input sequence $\{x(n)\}_{n=-\infty}^{\infty}$ that is bounded, it seems problematic in practice for $|y(n)|$ to increase without bound as $n\rightarrow\infty$.  Therefore, it is of interest to seek for properties of a system that preserve boundedness.  A system in which a bounded input yields a bounded output is called \textit{ BIBO stable}.  

In order to facilitate our discussion of filters, we will assume that our input sequences $\{x(n)\}_{n=-\infty}^{\infty}$ have the property that $x(n)=0$ for $n<0$.  A sequence with this property is known as \textit{causal}.  Moreover, we will assume that the input sequences are \textit{exponentially bounded}.  That is, we assume that 
\[
|x(n)|\leq K^{n},\enskip n\geq n_0
\]
for some constant $K$ and some integer $n_0$.  Now, to better understand the relationship between the input and output sequences, we make use of the following operator.  For any causal sequence $\{a(n)\}_{n=-\infty}^{\infty}$ that's exponentially bounded, consider the mapping $$\{a(n)\}_{n=-\infty}^{\infty}\mapsto\sum_{n=0}^{\infty}a(n)z^{-n}.$$  
This mapping is known as the $\textit{z-transform}$.  It is a linear operator from the space of exponentially bounded causal sequences onto the space of functions analytic at $\infty$. The $z$-transform of a sequence $\{a(n)\}_{n=-\infty}^{\infty}$ has a region of convergence (ROC) given by $z\in\hat{\C}$ such that 
\begin{equation}
\limsup\sqrt[n]{|a(n)|}<|z|. \label{ROC}   
\end{equation}
With this, the sum and product of two transformed sequences are defined to be in the intersection of both ROCs, and the product is given by the expression
\[
\Bigg(\sum_{n=0}^{\infty}a(n)z^{-n}\Bigg)\Bigg(\sum_{n=0}^{\infty}b(n)z^{-n}\Bigg)=\sum_{n=0}^{\infty}\Bigg(\sum_{k=0}^n a(k)b(n-k)\Bigg)z^{-n}.
\]

The $z$-transform has many properties that make it useful in the analysis of digital systems.  In particular, if $A(z)$ is the $z$-transform of the sequence $\{a(n)\}_{n=-\infty}^{\infty}$, then the $z$-transform of the sequence $\{c(n)=a(n-N): -\infty<n<\infty\}$ is $z^{-N}A(z)$ for any $N\in\N$.  
Therefore, by applying the $z$-transform to both sides of (\ref{diff}), we see that the $z$-transforms of the input and output are related by the equation
\begin{equation*}
    Y(z)=H(z)X(z),
\end{equation*}
where $H(z)$ is given by the rational function
\begin{equation}
H(z)=\frac{\sum_{k=0}^M b_k z^{-k}}{1+\sum_{j=1}^N a_j z^{-j}}.\label{trans}
\end{equation}
We use $X(z)$ and $Y(z)$ to represent the the $z$-transforms of the input and output sequences, respectively. The rational function $H$ is known as the transfer function, or \textit{filter}, of the system.  In the case of a recursive system, the transfer function is commonly called an \textit{infinite impulse response filter}, or IIR filter.

For the purpose of our discussion, we will only be considering recursive LTI systems.  For simplicity, we refer to a BIBO stable system as \textit{stable}.  Likewise, we refer to the filter of a BIBO stable system as stable.  The goal of Section 3 is to demonstrate how optimal polynomial approximants are used in designing a stable filter.  It's worth noting that although we will be considering filters of a single variable, several applications are concerned with filters of multiple variables. As mentioned in \cite{G}, the processing of medical pictures, satellite photographs, radar and sonar maps, seismic data mappings, gravity waves data, and magnetic recordings are examples in which 2D signal processing is needed. Here, we are concerned with designing filters $H(z,w)$ in $\C^2$ of analogous stability.

\subsection{Stability}
An important part of designing a digital filter is ensuring that the filter is stable.  The stability of the filter will prevent the magnitude of the output from increasing without bound, which could be damaging to the physical system.  Therefore, we seek for properties of the expression in (\ref{trans}) which guarantee stability.  If we define the polynomials $A(z)=1+\sum_{j=1}^N a_jz^j$ and $B(z)=\sum_{k=0}^Mb_k z^k$, we see that 
\begin{equation}
H(z)=\frac{z^{N-M}B^*(z)}{A^*(z)}, \label{conj}
\end{equation}
where $A^*(z)$ and $B^*(z)$ are the reverse polynomials of $A(z)$ and $B(z)$, respectively. From (\ref{conj}), it is easy to check that $H(z)$ is a rational function analytic at $\infty$. 
Therefore, $H(z)$ is the $z$-transform of some causal sequence $\{h(n)\}_{n=-\infty}^{\infty}$, and we write
\begin{equation}
    H(z)=\sum_{n=0}^{\infty}h(n)z^{-n}\label{series}
\end{equation}
for some specified ROC.

Now if we assume that the poles of $H(z)$ are contained in the disk $\D$, the expression in (\ref{series}) must be vaild on $\mathbb{T}$.  Since the series defined by $\sum_{n=0}^{\infty}|h(n)|z^{-n}$ has the same ROC, it follows that 
\begin{equation}
\sum_{n=0}^{\infty}|h(n)|<\infty.\label{sum}
\end{equation}
This observation leads us to the following theorem.

\begin{thm}
A filter $H(z)$ is stable if its poles are contained in the disk $\D$.
\end{thm}
\begin{proof}
Let $\{x(n)\}_{n=-\infty}^{\infty}$ be an input sequence with $|x(n)|\leq M$ for all $n$.  If the poles of $H(z)$ are contained in $\D$, then (\ref{sum}) holds.  Consequently, the output sequence $\{y(n)\}_{n=-\infty}^{\infty}$ is bounded with
\begin{align*}
    |y(n)|&=\Bigg|\sum_{k=0}^n h(k)x(n-k)\Bigg|\\
    &\leq M\sum_{n=0}^{\infty}|h(n)|
\end{align*}
for all $n$.  By definition, $H(z)$ must be stable.
\end{proof}
This theorem gives a sufficient condition for a filter to be stable.  However, it's important to note that a filter with a pole outside of $\D$ need not be stable.  As an example, consider the function $$H(z)=\frac{1}{z-1}.$$  For all $|z|>1$, this function is represented by the series $$H(z)=\sum_{n=1}^{\infty}z^{-n}.$$  If we consider the input defined by $x(n)=1$ for $n\geq 0$, the magnitude of the output is given by 
\begin{equation*}
    |y(n)|=\Bigg|\sum_{k=0}^nh(k)x(n-k)\Bigg|=n,
\end{equation*}
which clearly increases without bound as $n\rightarrow\infty$.

\subsection{Frequency Response}
Many problems are concerned with how a system responds to a sinusoidal input.  This is particularly evident in audio equalizing, where the input function represents a superposition of multiple sound waves.  Under the assumption that the system is linear, it is therefore advantageous to study the response of a system to the input $x(n)=e^{ins}$, where $s\in\R$ is a particular frequency.  In this case, the output is known as the \textit{frequency response} of the system.  If we let $H(z)$ be the corresponding filter, then the frequency response is expressed as 
\begin{align*}
    y(n)&=\sum_{k=0}^nh(k)e^{i(n-k)s}\\
    &=\bigg[\sum_{k=0}^nh(k)e^{-iks}\bigg]e^{ins}\\
    &=H(e^{is})e^{ins}.
\end{align*}
We then see that the frequency response is bounded, with \begin{equation*}
|y(n)|=|H(e^{is})|,\enskip n\geq 0.  
\end{equation*}
The quantity $|H(e^{is})|$ is referred to as the \textit{magnitude of the frequency response}. To get an idea of what this function looks like, consider the filter
\[
H(z)=\frac{0.3(z^2+2z+1)}{1.3z^2+1}.
\]
The poles and zeros of $H(z)$ are displayed in the following diagram:
\begin{center}
     \includegraphics[width=10cm]{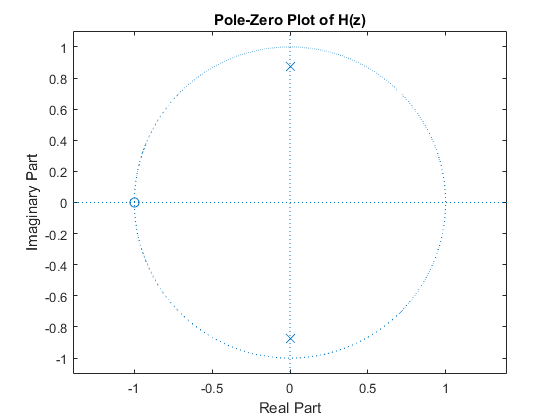}
\end{center}
The poles of $H(z)$ are marked with a cross and the zero of $H(z)$ is marked with a circle.  On the interval $[0,\pi]$, we therefore expect $|H(e^{is})|$ to have a maximum around 1.5 radians and a minimum around 3.1 radians.  This can be seen in the following graph: 
\begin{center}
     \includegraphics[width=10cm]{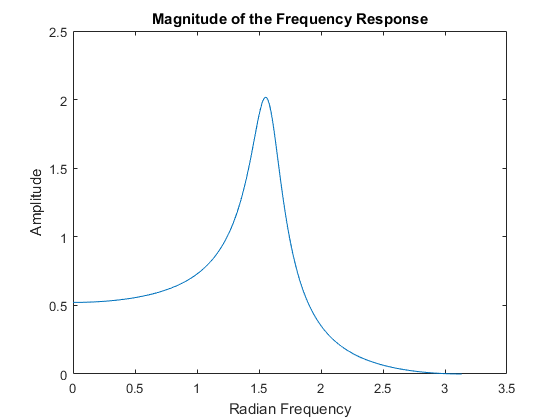}
\end{center}

Often in the design of a digital filter, the goal is to develop a rational function $H(z)$ in which the modulus satisfies a set of specifications on the boundary $\mathbb{T}$.  The effect of this would control the response of the system to the input $x(n)=e^{ins}$.  For the purpose of our discussion, we will assume that the specifications are given in the form of a non-negative even step function on $[-\pi,\pi]$.  Such a step function is known as an \textit{ideal digital filter}. 
The problem of digital filter design can then be stated as follows:

\begin{prob}\label{problem}
For a given ideal filter $\chi(e^{is})$, find a rational function $H(z)$ with poles inside of $\D$ such that $|H(e^{is})|$ is an approximation of $\chi(e^{is})$.
\end{prob}
Methods of approximation which guarantee stability of the filter is an interesting topic of research.  We present a method which has been modified from the ideas in \cite{CC}.  This method creates a rational function $p(z)/q(z)$, with poles in $\D$, such that $|p(e^{is})|/|q(e^{is})|$ approximates $\chi(e^{is})$ in the \textit{least-squares sense}.  For any $f\in L^2$ and any $\eta>0$, we say that the quotient $g/h$ of two functions in $L^2$ approximates $f$ in the least-squares sense if $$\|hf-g\|_{L^2}<\eta.$$  In this case, we call $g/h$ an \textit{(LS)-approximant} of $f$ and write $f\approx_{LS}g/h$.

We will present the method of approximation in three stages.  In the first stage, we will approximate the ideal filter by the magnitude of a non-vanishing function in $H^2$.  In the second stage, we will use optimal polynomial approximants to approximate this non-vanishing function with the magnitude of a rational function.  In the third stage, we will alter the numerator and denominator of the rational function in order to ensure stability. 



\subsection{First stage of approximation}

We start the first stage by defining a continuous function $\chi_{\varepsilon}(e^{is})$ in the following way.  Let $\mathcal{S}=\{s_j\}_{j=1}^N$ denote the points of discontinuity of $\chi(e^{is})$.  For each $s_j\in\mathcal{S}$, let $I_j=(s_j-\varepsilon/2,s_j+\varepsilon/2)$.  Here, $\varepsilon$ is a positive number chosen so that the intervals do not overlap and such that $\varepsilon$ is smaller than the minimum of the step values.  If $s\notin\cup_{j=1}^N I_j$, set
\[
\chi_{\varepsilon}(e^{is}):=
\begin{cases}
\chi(e^{is}) &\text{if $\chi(e^{is})>0$}\\
\varepsilon &\text{if $\chi(e^{is})=0$}.
\end{cases}
\]
This creates a positive step function on $[-\pi,\pi]\setminus\cup_{j=1}^N I_j$.  Then connect each successive step with a straight line segment.  For each $s\in\cup_{j=1}^N I_j$, set $\chi_{\varepsilon}(e^{is})$ to coincide with these segments.  This creates a non-vanishing continuous function on $[-\pi,\pi]$.

We then create an analytic function on $\D$ by using $\chi_{\varepsilon}(e^{is})$.  For any $z\in\D$, define the function
\begin{equation*}
    f_{\varepsilon}(z):=\exp\left(\frac{1}{2\pi}\int_{-\pi}^{\pi}\frac{e^{is}+z}{e^{is}-z}\log\chi_{\varepsilon}(e^{is})ds\right).
\end{equation*}
Note that $f_{\varepsilon}(z)$ is analytic in $\D$, non-vanishing in $\overline{\D}$, and has the property that 
\[
\log |f_{\varepsilon}(z)|=\frac{1}{2\pi}\int_{-\pi}^{\pi}\operatorname{Re}\left(\frac{e^{is}+z}{e^{is}-z}\right)\log\chi_{\varepsilon}(e^{is})ds.
\]
i.e., $\log |f_{\varepsilon}(z)|$ solves the Dirichlet problem in $\D$ with boundary values defined by $\log\chi_{\varepsilon}(e^{is})$.  Therefore, the analytic function satisfies 
\begin{equation}
|f_{\varepsilon}(e^{is})|=\chi_{\varepsilon}(e^{is})\label{bound}
\end{equation}
for all $s\in[-\pi,\pi]$.  This leads us to the following theorem.
\begin{thm}
Any ideal filter $\chi$ can be approximated (in the least-squares sense) by a non-vanishing function $f\in H^2$.
\end{thm}
\begin{proof}
Let $\chi(e^{is})$ be an ideal filter with a collection of discontinuities $\mathcal{S}=\{s_j\}_{j=1}^N$.  Given any $\eta>0$, choose $\varepsilon$ to satisfy 
\[
0<\varepsilon<\min\bigg\{\frac{\eta^2\pi}{N(\|\chi\|_{\infty}^2+1)},\frac{\eta}{\sqrt{2}}\bigg\}.
\]
Let $\{E_k\}_k$ denote the collection of intervals for which $\chi_\varepsilon(e^{is})=\varepsilon$. From (\ref{bound}), it follows that
\begin{align*}
    \|\chi-|f_{\varepsilon}|\|^2_{L^2}&=\|\chi-\chi_{\varepsilon}\|^2_{L^2}\\
    &=\sum_{j=1}^N\frac{1}{2\pi}\int_{I_j}|\chi(e^{is})-\chi_{\varepsilon}(e^{is})|^2ds\\ &+\sum_{k}\frac{1}{2\pi}\int_{E_k}|\chi(e^{is})-\chi_\varepsilon(e^{is})|^2ds\\
    &\leq \varepsilon\frac{N}{2\pi}\|\chi\|_{\infty}^2 +\varepsilon^2\sum_k\frac{1}{2\pi}\int_{E_k}ds\\
    &\leq\varepsilon\frac{N}{2\pi}\|\chi\|_{\infty}^2+\varepsilon^2\\
    &<\eta^2.
\end{align*}
\end{proof}

This theorem states that $|f_{\varepsilon}(z)|$ is an (LS)-approximant of $\chi(e^{is})$.  Now, since $f_{\varepsilon}(z)$ is a function in $H^2$ that doesn't vanish at the origin, the $n$-th optimal polynomial approximant $q_n$ of $1/f_{\varepsilon}$ is non-vanishing on $\overline{\D}$ (see, e.g., Theorem \ref{minimal_zero}).  This suggests that $q_n$ (more specifically, the reverse polynomial of $q_n$) should be a part of our rational function $H(z)$.  This observation leads to the second stage of the approximation.

\subsection{Second stage of approximation}\label{APTSD}

Given an ideal filter $\chi(e^{is})$, the first stage of the approximation involved determining the function $f_{\varepsilon}(z)$.  It then followed that $\chi(e^{is})\approx_{LS}|f_{\varepsilon}(e^{is})|$.  In the next stage, we approximate $f_{\varepsilon}$ with the magnitude of a rational function. 



Since $f_{\varepsilon}$ is actually an outer function in $H^2$, it follows that $\|q_nf_{\varepsilon}-1\|_{L^2}\rightarrow 0$
as $n\rightarrow\infty$, where $q_n$ denotes the $n$-th opa of $1/f_{\varepsilon}$.  Therefore, given any $\eta>0$, we can choose $N$ so that 
$\|q_Nf_{\varepsilon}-1\|_{L^2}<\eta.$  Moreover for any $M\geq 0$, we see that
\[
\|q_Nf_{\varepsilon}-p_M\|_{L^2}=
\inf_{p\in\mathcal{P}_M}\|q_Nf_{\varepsilon}-p\|_{L^2}<\eta,
\]
where $p_M$ denotes the orthogonal projection of $q_Nf_{\varepsilon}$ onto $\mathcal{P}_M$.  Hence, we have that $p_M/q_N$ is an (LS)-approximant of $f_{\varepsilon}$.  Consequently, we have that $|f_{\varepsilon}(e^{is})|\approx_{LS}|p_M(e^{is})|/|q_N(e^{is})|$.

It's important to note that it's computationally efficient to determine the polynomials $p_M$ and $q_N$.  We have already seen in Theorem \ref{LinearSys} that the coefficients of $q_N$ can be expressed as the solution of a system of $N+1$ linear equations, each of which are dependent only on the function $f_{\varepsilon}$.  The entries of the associated matrix $B$ can be expressed as a Fourier coefficient of the $L^2$ function $|f_{\varepsilon}|^2$, i.e., 

\begin{align*} 
	  B_{jk}=\frac{1}{2\pi}\int_{-\pi}^{\pi}|f_{\varepsilon}(e^{is})|^2e^{i(k-j)s}ds. 
\end{align*} 
Hence, they can be computed efficiently through any available FFT algorithm. Furthermore, $B$ is a Gram matrix generated by the vectors $\{z^kf_{\varepsilon}\}_{k=0}^N$. Since these vectors are linearly independent, it follows that $B$ is invertible. Moreover, since Gram matrices are positive definite, and since $B$ is Hermitian and Toeplitz, we can use any of the fast algorithms to compute its inverse. We then see that the coefficients of $q_N$, say $a_0,\dots,a_N$, are given by the expression 
\[
\begin{bmatrix}
a_0\\
a_1\\
\vdots\\
a_N
\end{bmatrix}
=B^{-1}
\begin{bmatrix}
\overline{f_{\varepsilon}(0)}\\
0\\
\vdots\\
0
\end{bmatrix}.
\]
Hence, the coefficients are determined by the first column of $B^{-1}$ scaled by $\overline{f_{\varepsilon}(0)}$.  On the other hand, the coefficients of $p_M$ are given by the first $M+1$ Fourier coefficients of $q_Nf_{\varepsilon}$.

\subsection{Third stage of approximation}
In the first two stages of approximation, we were able to approximate an ideal filter $\chi(e^{is})$ with the magnitude of a rational function.  More specifically,
\[
\chi(e^{is})\approx_{LS}|f_{\varepsilon}(e^{is})|\approx_{LS}\frac{|p_M(e^{is})|}{|q_N(e^{is})|}.
\]
Since the ideal filter is assumed to be an even function on $[-\pi,\pi]$, we have that 
\begin{align*}
    \chi(e^{is})&=\chi(e^{-is})\\
    &\approx_{LS}|f_{\varepsilon}(e^{-is})|\\
    &\approx_{LS}\bigg|\frac{p_M(e^{-is})}{q_N(e^{-is})}\bigg|\\
    &=\bigg|\frac{p_M^*(e^{is})}{q_N^*(e^{is})}\bigg|,
\end{align*}
where $p_M^*$ and $q_N^*$ are the reverse polynomials of $p_M$ and $q_N$, respectively.  Now, since $M$ is arbitrary, choose $M\leq N$ and define the rational function
\begin{equation}
    H(z)=\frac{p^*_M(z)}{q^*_N(z)}. \label{filter}
\end{equation}
Then $H(z)$ is a rational function that's analytic at $\infty$.  It has the property that $|H(e^{is})|$ is an approximation of $\chi(e^{is})$.  Furthermore, since the zeros of $q_N$ are outside of $\overline{\D}$, and since $H(z)$ is analytic at $\infty$, it follows that the poles of $H(z)$ are contained in the disk $\D$.  Therefore, the expression in (\ref{filter}) gives us our stable filter.

This three stage method of approximation stems from the ideas presented in 1982 by Chui and Chan in \cite{CC}.  The use of opas in filter design doesn't seem to have gone much further in the one variable case, although some papers later in the 80s and 90s discuss related IIR filter designs (see, e.g., \cite{CCP,CPC,DC}).  It would be interesting to know whether opas might have some further applications in signal processing research.  There do remain some open problems in the several variable case, as discussed in Section 4.2.


\section{Zeros of opas}\label{zeros}

\subsection{One variable spaces}

Let us now turn to a discussion of zeros of opas in the spaces $H^2_\omega.$ According to Theorem \ref{H2OP}, given $f \in H^2$, the $n$-th opa $q_n$ is equal to a multiple of the $n$-th reverse orthonormal polynomial $\varphi_n^*$. Assuming $f(0)\neq 0,$ it is well-known (see, e.g., \cite{Ge,BaS}) that $\varphi_n$ has all its zeros inside the open unit disk $\D,$ and therefore $\varphi_n^*$, and hence $q_n$, whose zeros are reflected across the unit circle, have no zeros inside the closed unit disk. From an engineering perspective, for a given function $f$, the produced filter whose poles coincide with the zeros of $\varphi_n$ is a stable filter. The authors of \cite{blSimanek} studied this phenomenon for the more general spaces $H^2_\omega$.  They considered the following ``minimal-zero" problem: what is the infimum of the modulus of any zero of an optimal polynomial approximant $q_n$ for any $n$ and any $ f \in H^2_\omega$ (assuming $f(0) \neq 0$)? In other words, can any zero of any opa (for any function $f$) penetrate the disk for a given space $H^2_\omega$ and if so, how far inside the disk can that zero appear? The authors proved the following result (see Theorem 2.1 in \cite{blSimanek}).

\begin{thm}[\cite{blSimanek}]\label{minimal_zero}
Given $f\in H^2_\omega$ such that $f(0)\neq 0,$ for each $n \in \N,$ let $q_{n,f}$ denote the $n$-th opa of $1/f.$ Consider the extremal problem of identifying
\begin{equation*}
   M:= \inf \left\{ |z|: q_{n,f}(z)=0 \mbox{ for some }  f \in H^2_\omega \mbox{ and some } n \in \N \right\}.
\end{equation*}
Then if the sequence of weights $\{\omega_k\}_{k \in N}$ is non-decreasing, $M = 1$ and the infimum is not achieved.  On the other hand, if the sequence of weights is such that there exist $k,n \in \N$ that satisfy $\omega_{k+n+1}< \omega_{k+1}/4,$ then $M < 1$ and the infimum is a minimum.  
\end{thm}

In particular, for any of the Dirichlet-type spaces $D_{\alpha},$ if $\alpha \geq 0$ (for example, in $H^2$ or in the classical Dirichlet space) the infimum is $1$, and all the zeros of the optimal approximants lie outside the closed unit disk.  On the other hand, if $\alpha < 0,$ there exist functions $f \in D_{\alpha}$ whose optimal polynomial approximants have zeros inside the unit disk. In particular, in the Bergman space $A^2,$ the authors showed (see Theorem 5.1 of \cite{blSimanek} that $M = 2 \sqrt{2}/3$ and the corresponding extremal function is 
$f(z) = \frac{1}{(1-z/\sqrt{2})^3}.$

The proof of the existence of an extremal in the case that the weights satisfy the required inequality in Theorem \ref{minimal_zero} from \cite{blSimanek} is difficult and relies on machinery from the theory of orthogonal polynomials. Below, we give an outline of the argument involved in the easier parts of the proof and only mention the general ideas in the proof of existence of the extremal. 

\begin{proof}
First note that it is enough to consider $n=1$, since if $z_1$ is a zero of some optimal approximant $q_n$ ($n > 1$) for some function $f$, then since
$$\|q_nf-1\|_\omega= \|(z-z_1) \frac{q_n}{z-z_1} f - 1\|_\omega,$$
$z_1$ is a zero of a first order approximant for some function that is a multiple of $\frac{q_n}{z-z_1} f.$

In that case, if $z_1$ is a zero of a first order approximant for some function $f$, it is not hard to see that 
\begin{equation*}
    z_1 = \frac{\|zf\|_\omega^2}{\langle f, zf \rangle_\omega},
\end{equation*}
and therefore the extremal problem becomes to find
\begin{equation}\label{EP}
    M = \inf \left\{ \frac{\|zf\|_\omega^2}{|\langle f, zf \rangle_\omega|}, f \in H^2_\omega \right\}.
\end{equation}
Using the Cauchy-Schwarz inequality and noticing that if the sequence $\{\omega_k\}_{k \in N}$ is non-decreasing, then $\|zf\|_\omega \geq \|f\|_\omega$ gives that for those weights, $M \geq 1.$ Moreover, the condition on the weights given in the introduction ensure that $M \leq 1.$  Therefore $M = 1$ in that case, and it is not hard to see that equality cannot hold in the Cauchy-Schwarz inequality, which implies that the infimum is not attained. 

On the other hand, if the weights satisfy the condition that $\omega_{k+n+1}< \omega_{k+1}/4,$ the authors of \cite{blSimanek} show that if 
$$f(z) = z^k T_n \left( \frac{1+z}{1-z} \right),$$ where $T_n(g)$ is the $n$-th Taylor polynomial of a function $g$, then the ratio 
$$\frac{\|zf\|^2_\omega}{|\langle f, zf \rangle_\omega|}$$ is strictly less than $1$, by explicit computation.

The remainder of the proof involves reducing the extremal problem \eqref{EP} to one that only involves functions $f$ that are polynomials with positive coefficients and of degree at most $N$. If $Q_N$ is the corresponding extremal polyomial, the authors then show that the coefficients of these extremal polynomials $Q_N$ satisfy a three-term recurrence relation, and are therefore orthogonal polynomials on the real line, which in turn are connected to Jacobi matrices and associated homogeneous linear differential equations. The authors show that if $J$ is the Jacobi matrix with entries 
$J_{ij} = \sqrt{\frac{\omega_j}{\omega_{j+1}}}$ if $|i-j|=1$ and $J_{ij}=0$ otherwise, then $M=\frac{2}{\|J\|},$ where $\|J\|$ is the norm of the matrix $J.$ The authors identify the extremal function in terms of certain orthonormal polynomials associated with $Q_N$ and with the norm of the matrix $J$ (see Corollary 4.5 of \cite{blSimanek}).  
\end{proof}

In general, it is very difficult to calculate $M$ explicitly, and the authors investigate this question in detail, in particular in some other spaces involving integral norms, but in particular, the following question remains open.

\begin{OQ}\label{OpenQuestion1}
    What is $M$ for the spaces $D_{\alpha}$ when $ \alpha < 0,$ $\alpha \neq -1?$ What is the corresponding extremal function?  Given the connection with Jacobi matrices above, this question can be rephrased as, for $\alpha < 0,$ $\alpha \neq -1,$ what is the norm of the Jacobi matrix $J$ defined by $J_{ij} = \sqrt{\frac{(j+1)^{\alpha}}{(j+2)^{\alpha}}}$ if $|i-j|=1$ and $J_{ij}=0$  otherwise?
\end{OQ}
One may also ask the corresponding question for a \emph{fixed} integer 
$n > 1.$ In addition, there are many interesting questions surrounding the extremal function, which satisfies a differential equation and is equal to the derivative of the reproducing kernel evaluated at a particular point of the disk, for certain spaces, which seems to warrant further investigation (see Section 8 of \cite{blSimanek}).

Another direction of inquiry investigates the limit points of the zeros of the optimal polynomial approximants.  Indeed, it turns out that for a general class of weighted spaces where the weights are associated with certain ``regular measures", and for (say) cyclic functions $f$ such that $1/f$ has a singularity on the unit circle, every point of the unit circle is a limit point of the zeros of the optimal polynomial approximants of $1/f.$  This is an analogue of a beautiful theorem of Jentzsch that states that if an analytic function in the unit disk has radius of convergence $1$, then every point on the unit circle is a limit point of the zeros of the Taylor polynomials of that function.  The proof of the theorem follows the outline of the original proof of Jentzsch's theorem, but knowledge of the precise asymptotic behavior of the orthogonal polynomials for regular measures connected to the opas is required. For details, see \cite[Theorem 6.2]{blSimanek}. In fact, the authors prove that asymptotically, for each $\varepsilon > 0,$ the zeros of the $n$-th opa for such a function $f$ lie in a disk of radius $1+\varepsilon.$ (See Theorem 6.1 in \cite{blSimanek}.)  However, more detail on the precise behavior of the zeros is needed. For instance, the following question is open. 

\begin{OQ}\label{OpenQuestion2}
Given a cyclic function $f \in H^2_{\omega},$ and given $\zeta \in \T,$ is there a sector with vertex at $\zeta$ that is devoid of zeros of opas?  More generally, are there regions or rays that the zeros avoid?
\end{OQ}

Finally, another line of investigation related to zeros involves cyclicity.  In \cite[Theorem 6.1]{BKLSS}, the authors find a characterization of cyclicity of a function $f \in H^2$ based on the relationship between the zeros of the opas and the value of $f$ at the origin.  Since cyclic functions in $H^2$ are outer functions and are well-understood, this characterization is not that useful for $H^2,$ but perhaps this idea can be extended to other spaces.  Indeed, one might hope that such a characterization would give insight into the Brown and Shields Conjecture.  Thus, although vague, the following is open. 

\begin{OQ}\label{OpenQuestion3}
For a given space $H^2_{\omega}$ (or, say for the Bergman space $A^2$ or the classical Dirichlet space $\mathcal{D}$), is there a characterization of a cyclic function $f$ that relies on the behavior of the zeros of the opas for $1/f$?
\end{OQ}

\subsection{Several variable spaces and Shanks Conjecture}

When considering Hilbert spaces of analytic functions of several variables, the zeros of opas are less well-understood, partly because zeros of functions of several variables are no longer isolated.  On the other hand, this topic was of great interest in the engineering literature of the 70s in connection with filters, as discussed earlier.  Two recent papers (see \cite{SS,SS2}) discuss the several variable situation in detail, so we just mention a few relevant items of interest here. 

For simplicity, let us consider the Dirichlet-type spaces of the bidisk defined as follows.  Let $(\alpha_1,\alpha_2) \in \R^2.$ Given $$f(z_1,z_2) = \sum_{j=0}^{\infty} \sum_{k=0}^{\infty} a_{j,k}z_1^jz_2^k$$ analytic in $\D^2,$ we will say $f \in \mathcal{D}_{\alpha_1,\alpha_2}$ if 
\[ \|f\|_{\alpha_1,\alpha_2}^2 := \sum_{j=0}^{\infty} \sum_{k=0}^{\infty} (j+1)^{\alpha_1}(k+1)^{\alpha_2} |a_{j,k}|^2 < \infty.\]  This makes $\mathcal{D}_{\alpha_1,\alpha_2}$ into a reproducing kernel Hilbert space.  In order to define optimal polynomial approximants, as the authors of \cite{SS} note, one needs to order the monomials $z_1^jz_2^k$ in some fashion. Let us assume that we have chosen such an ordering (for instance, the degree lexicographic ordering), and let $\mathcal{P}_n$ be the span of the first $n+1$ monomials.  In this way, given $f \in \mathcal{D}_{\alpha_1,\alpha_2},$ not identically $0$, we can define, as before, the $n$-th optimal polynomial approximant of $1/f$ in $\mathcal{D}_{\alpha_1,\alpha_2}$ to be the polynomial $q_n$ that minimizes
$\|pf-1\|_{\alpha_1,\alpha_2}$ among all polynomials $p \in \mathcal{P}_n$.

In \cite{BCLSS2}, the authors extended their results from \cite{BCLSS} to the case when $\alpha_1= \alpha_2 = \alpha$ to obtain rates of decay for $\|q_nf - 1\|_{\alpha,\alpha}$ for functions $f$ analytic in the closed unit bidisk with no zeros in the bidisk. They also gave examples of polynomials with no zeros on the bidisk that are not cyclic in $\mathcal{D}_{\alpha,\alpha}$ for $\alpha > 1/2$.  Cyclic polynomials in these Dirichlet spaces of the bidisk (for $\alpha_1= \alpha_2$) were completely characterized in \cite{BKKLSS}.  In \cite{KKRS}, the authors characterized cyclic polynomials in the anisotropic Dirichlet spaces, that is, the more general case of $\mathcal{D}_{\alpha_1,\alpha_2}$ when $\alpha_1$ and $\alpha_2$ may be different. 

In \cite{SS}, the authors discuss the history of the problem of using optimal polynomial approximants to design digital filters of two variables.  In this quest, in \cite{STJ}, the authors conjectured that if $f$ is a polynomial that is zero-free in the bidisk, then optimal polynomial approximants of $1/f$ in the Hardy space of the bidisk would also be zero-free in the bidisk.  This became known as Shanks Conjecture, which was disproved in \cite{GK}. A simplified counter-example can be found in \cite{SS}.  However, the following weaker version of the Shanks conjecture remains open.   

\begin{OQ}\label{Shanks}[Weak Shanks Conjecture]
Let $f(z_1,z_2)$ be a polynomial that doesn't vanish in $\overline{\D^2}$.  Then the optimal polynomial approximant of $1/f$ in $H^2(\D^2)$ is zero-free in $\D^2$.
\end{OQ}

It's interesting to know that although this result is unknown in $H^2(\D^2)$, it does fail in other function spaces of the bidisk, including the Bergman space $A^2(\D^2)$ (see Example 22 of \cite{SS}).

\section{Convergence Results} 

One set of questions of great interest but with little known so far is about rate of convergence of opas for a given cyclic function $f$. Since $f$ is cyclic, we know that $\|q_nf-1\|_{\omega}$ approaches $0$ at the fastest possible rate in terms of the degree.  What is that rate of convergence, for a given $f$?  In addition, since norm convergence implies uniform convergence on compact subsets of the open unit disk, we know that $q_nf$ approaches $1$ pointwise, uniformly on compact subsets of $\D,$ but what more can be said about convergence on the circle?  In what follows, we discuss what is known about these questions and some related open problems. 

\subsection{Norm Convergence}

In \cite{BCLSS}, the authors studied the rate of decay of $\|q_nf-1\|_{\alpha}$ for certain simple cyclic functions $f$ in the Dirichlet-type spaces $D_{\alpha}$ for $\alpha \leq 1.$ Note that for $\alpha >1,$ a function is cyclic if and only if it does not vanish in the closed disk. Also, if $f$ is such that $1/f$ is analytic in the closed disk, then it is easy to see that the rate of decay of $\|q_nf-1\|_{\alpha}$ is exponential. Therefore the question of rate of decay is most interesting for $\alpha \leq 1$ and for functions $f$ such that $1/f$ has a singularity on the unit circle.  The simplest possible function to consider then is $f(z) = 1-z.$

For the function $1-z,$ as we saw in Example \ref{1-z}, an explicit formula for $q_n$ can be obtained, and moreover, the authors of \cite{BCLSS} calculated the rate of decay of $\|q_nf-1\|_{\alpha}$ and showed that this rate of decay does not change for \emph{any} polynomial whose zeros are outside the unit disk (with at least one zero on the circle). They then extended that result to any function that admits an analytic extension to the closed unit disk.  More specifically, their result is the following. 

\begin{thm}\label{rate_convergence}
Let $\alpha \leq 1$ and let $f$ be analytic in the closed unit disk and have zeros outside $\D.$  Then for each $n \in \N$ there exists a constant $C$ independent of $n$ such that
\begin{equation*}
    \|q_nf-1\|^2_{\alpha} \leq \begin{cases}
\frac{C}{(n+1)^{1-\alpha}} \,\, \mbox{for } \alpha < 1 \\
\frac{C}{\log^+(n+1)} \,\, \mbox{for } \alpha =1.
\end{cases}
\end{equation*}
Moreover, if $f$ has a zero on the unit circle, this rate of convergence is sharp.
\end{thm}

In \cite{FMS}, the authors showed that the rate of decay of 
$\|q_nf-1\|_{\omega}^2$ for $f(z) = 1-z$ is precisely equal to 
$\left( \frac{1}{\sum_{k=0}^{n+1} \frac{1}{\omega_k}} \right)$; the authors of \cite{SS} then exploited that single-variable rate to identify rates of decay for various analogues of $1-z$ in some of the several variable spaces. In \cite{RSS}, the authors came up with an example of a lacunary series of the type $f(z) = 1 + \sum_{k=1}^{\infty} a_k z^{2^k}$ whose corresponding rate of decay in the Dirichlet-type spaces $\mathcal{D}_{\alpha}$ is \emph{slower} than the rate in  Theorem \ref{rate_convergence}.  For instance, they construct a lacunary function $f$ in $H^2$ such that $1/(\log n)^{1+\varepsilon}$ is a lower bound for the rate of decay of 
$\|q_nf-1\|_{H^2}^2$.   However, for most other functions, the rate of decay of $\|q_nf-1\|_{\alpha}$ is unknown, and thus, the following general problem  is essentially open for most functions. 

\begin{OQ}\label{OQ_RateConvergence}
Given a function $f \in D_{\alpha}$ such that $f$ does not have an analytic extension to the closed disk, find the rate of decay of $\|q_nf-1\|_{\alpha}$.
\end{OQ}

In \cite{Ch}, the author considered a slightly different convergence question in the Hardy space $H^2$ related to double least squares. Given a \emph{polynomial} $f$ of degree $n$, for each $k \in \N,$ let $q_k$ be the $k$-th optimal polynomial approximant (in $H^2$) of $1/f.$ Then let $Q_{n,k}$ be the $n$-th optimal polynomial approximant of $1/q_k.$ Since we expect $q_k$ to be some kind of approximation of $1/f$ and $Q_{n,k}$ some kind of approximation of $1/q_k,$ we expect $Q_{n,k}$ to approximate $f$ as $k \rightarrow \infty.$  Indeed, this will be the case (as shown in \cite{Ch}) for cyclic polynomials $f$ (i.e., ones that have no zeros in $\D.$)  In fact, Chui showed that the limit of the polynomials $Q_{n,k}$ is a polynomial of degree $n$ that preserves the zeros of $f$ that are outside $\D$ and contains in addition the reflection of the zeros of $f$ that are inside $\D.$  We summarize these two results of Chui in the following theorem.

\begin{thm}[\cite{Ch},Theorems 2.1 and 3.1] 
Let $f(z) = p(z) \cdot \prod_{j=1}^m (\alpha_j - z)$ be a polynomial of degree $n$ where $\alpha_j\in \D\setminus\{0\}$ and $p$ is a polynomial of degree $n-m$ that has no zeros in $\D.$ Let $q_k$ be the $k$-th optimal polynomial approximant in $H^2$ of $1/f,$ and let $Q_{n,k}$ be the $n$-th optimal polynomial approximant of $1/q_k.$ Then as $k \rightarrow \infty,$
$\|Q_{n,k} - \tilde{f}\|_{H^2} \rightarrow 0,$ where 
$$\tilde{f}(z) = p(z) \cdot \prod_{j=1}^m (1/\bar{\alpha_j} - z).$$  As a consequence, $\|Q_{n,k} - f\|_{H^2} \rightarrow 0$ if and only if $f$ is a cyclic polynomial, i.e., has no zeros in $\D.$ 
\end{thm}

Izumino (see \cite{Iz}) extended Chui's results to functions 
$f \in H^{\infty}$ such that $1/f \in H^{\infty}$ using operator theory methods. He also proved a conjecture that Chui had considered related to
double least squares of \emph{lower} degree than the degree of the original polynomial $f$. More specifically, he proved the following.
 
\begin{thm}[Theorem 3.5 of \cite{Iz}]
Let $f(z) = p(z) \cdot \prod_{j=1}^m (\alpha_j - z),$ where $|\alpha_j|=1$ for $j=1, \ldots, m$ and $p \in H^{\infty}$ is an outer function. Then for each $n=0, \ldots, m-1,$ $\|Q_{n,k} \|_{H^2} \rightarrow 0$ as 
$k \rightarrow \infty.$
\end{thm}

Notice that the convergence result of Chui's is a result in a finite dimensional space, the space of polynomials of degree at most $n,$ and therefore convergence of the \emph{double} least-squares also happens for instance uniformly in the closed disk.  In general, though, if $f$ is a cyclic function, we know $(q_nf-1)(z)$ converges to $0$ on compact subsets of the open unit disk, because of the norm convergence, but what happens on the unit circle? For a ``good" function, should we expect that $q_n(\zeta) \rightarrow 1/f(\zeta)$ for $\zeta \in \T$? This leads to another series of questions that concern pointwise convergence of optimal polynomial approximants on the unit circle, which we discuss in the next section. 

\subsection{Pointwise convergence}

The boundary behavior of optimal polynomial approximants depends heavily on the function $f$ whose inverse they are approximating. In two recent papers (\cite{BMS2, BMS}), the authors investigate two distinct phenomena. In \cite{BMS2}, the authors show that there exist many $H^2$ functions $f$ whose opas on the unit circle can approximate any complex number. This is a kind of universality phenomenon.  More precisely, they proved a theorem that implies the following.

\begin{thm}[\cite{BMS2}]\label{universality}
For any $\zeta \in \T,$ there exists a $G_{\delta}$-dense set of functions $f$ in $H^2$ (or in the classical Dirichlet space) with corresponding optimal approximants $q_n$ such that the set of points $\{q_n(\zeta)\}$ is dense in $\C.$ 
\end{thm}
Thus clearly it is not the case that $q_n(\zeta)$ converges to $f(\zeta)$ for that particular point and for that large set of functions $f$.
However, if $f$ is well-behaved enough, pointwise convergence of the opas at points on the circle that are not zeros of $f$ does occur.  In fact, even more is true, as seen in the following.  

\begin{thm}[\cite{BMS}]\label{BMS_convergence}
Let $f$ be a polynomial with simple zeros that lie outside the open unit disk, and let $q_n$ be the $n$-th opa of $1/f$ in the Hardy space $H^2$ or the Bergman space $A^2.$ Then $1 - q_n f$ converges to $0$ uniformly on compact subsets of $\overline{\D} \setminus Z(f)$, where $Z(f)$ is the set of zeros of $f$.
\end{thm}
The proof of this result is somewhat technical and relies on estimates of the coefficients of $1 - q_n f$ which are tractable in the Hardy and the Bergman space.  It seems likely based on preliminary work that the result holds for any polynomial with no zeros in $\D$, but the computations become more complicated. It would be interesting to know whether these results still hold for more general spaces and how ``good" the functions have to be in order to get convergence versus the universality behavior observed in Theorem \ref{universality}.  Thus we conclude with the following open questions. 

\begin{OQ}\label{convergence_pols} 
Suppose $f$ is a polynomial with zeros $Z(f)$ in the complement of $\D$ and let $q_n$ be the $n$-th opa of $1/f$ in $H^2_{\omega}.$  For which weights $\{ \omega_k \}_{k \in \N}$ is it true that $1 - q_n f$ converges to $0$ uniformly on compact subsets of $\overline{\D} \setminus Z(f)$? 
\end{OQ}

\begin{OQ}\label{convergence_goodfn}
Does Theorem \ref{BMS_convergence} hold for functions other than polynomials, and if so, for which ones? On the other hand, can we characterize the functions for which the universality behavior as in Theorem \ref{universality} occurs?
\end{OQ}


\end{document}